\documentclass[a4paper, 11pt]{amsart}
\usepackage{amssymb,array}
\usepackage{url}
\usepackage{graphicx}
\usepackage{subfigure}

\usepackage[left=3cm, right=3cm, top=3cm, bottom=3cm]{geometry}

\renewcommand{\leq}{\leqslant}
\renewcommand{\geq}{\geqslant}

\newcommand{\N}{\mathbb{N}}

\newcommand{\R}{\mathbb{R}}
\newcommand{\C}{\mathbb{C}}
\newcommand{\E}{\mathbb{E}}
\renewcommand{\P}{\mathbb{P}}

\newcommand{\K}{\mathcal{K}}

\newcommand{\D}{\mathcal{D}}
\newcommand{\M}{\mathcal{M}}

\renewcommand{\S}{\mathfrak{S}}

\newcommand{\module}[1]{\left| #1\right|}
\newcommand{\floor}[1]{\lfloor #1 \rfloor}

\newcommand{\norm}[1]{\left\Vert #1\right\Vert}

\newcommand{\ol}{\overline}

\newcommand{\Bell}{\mathrm{Bell}}

\DeclareMathOperator{\im}{Im}
\DeclareMathOperator{\trace}{Tr}

\DeclareMathOperator{\I}{I}

\DeclareMathOperator{\Gr}{Gr}
\DeclareMathOperator{\rk}{rk}
\DeclareMathOperator{\spec}{spec}
\DeclareMathOperator{\End}{End}

\newcommand{\scalar}[2]{\langle #1 , #2\rangle}
\newcommand{\e}{\varepsilon}
\renewcommand{\phi}{\varphi}
\newcommand{\iy}{\infty}

\newcommand{\ketbra}[2]{| #1 \rangle \langle #2 |}

\newtheorem{theorem}{Theorem}[section]

\newtheorem{proposition}[theorem]{Proposition}
\newtheorem{remark}[theorem]{Remark}
\newtheorem{lemma}[theorem]{Lemma}

\newtheorem{corollary}[theorem]{Corollary}

\begin{document}

\title[Random quantum channels II]{Random quantum channels II: Entanglement of random subspaces, 
R\'enyi entropy estimates and additivity problems}
\author{Beno\^{\i}t Collins}
\address{
D\'epartement de Math\'ematique et Statistique, Universit\'e d'Ottawa,
585 King Edward, Ottawa, ON, K1N6N5 Canada
and 
CNRS, Institut Camille Jordan Universit\'e  Lyon 1, 43 Bd du 11 Novembre 1918, 69622 Villeurbanne, 
France} 
\email{bcollins@uottawa.ca}
\author{Ion Nechita}
\address{
Institut Camille Jordan Universit\'e  Lyon 1, 43 Bd du 11 Novembre 1918, 69622 Villeurbanne, 
France} 
\email{nechita@math.univ-lyon1.fr}
\subjclass[2000]{15A52, 94A17, 94A40} 
\keywords{Random matrices, Quantum information theory, Random channel, Additivity conjecture, Schatten p-norm}

\begin{abstract}
In this paper we obtain new bounds for the minimum output entropies of random quantum channels. These bounds rely on random matrix techniques arising from free probability theory.
We then revisit the counterexamples developed by Hayden and Winter to get violations of the additivity equalities for minimum output R\'enyi entropies. We show that random channels obtained by randomly coupling the input to a qubit violate the additivity of the $p$-R\'enyi entropy, for all $p>1$. For some sequences of random quantum channels, we compute almost surely the limit of their Schatten $S_1 \to S_p$ norms.
\end{abstract}

\maketitle

\section{Introduction}

The relationship between random matrix theory and free probability theory
 lies in the asymptotic freeness of random matrices. 
Asymptotic freeness, as it was discovered by Voiculescu
(see e.g. \cite{voiculescu-dykema-nica}), usually 
predicts the asymptotic pointwise behavior of joint non-commutative moments. 
However in some cases it can also predict more. For example, in the
case of i.i.d. GUE random matrices, it was showed by Haagerup and
Thorbj\o{}rnsen \cite{haagerup-thorbjornsen} that even the norms have an almost sure 
behavior predicted by free probability theory.

Quantum information theory is the analogue of 
classical information theory, where classical communication protocols 
are replaced by quantum information protocols, known as 
quantum channels. 
Despite the apparent simplicity of some mathematical question related to 
information theory, their resistance to various attempts to (dis)prove them
have led to the study of their statistical properties.

The Holevo conjecture is arguably the most important conjecture in quantum information theory,
and the theory of random matrices has been used here with success
by Hayden and Winter \cite{hayden,hayden-winter}
to produce counterexamples to the
additivity conjecture of R\'enyi entropy for $p>1$.
These counterexamples are of great theoretical importance, as they
depict the likely behavior of a random channel. 
In \cite{hastings}, Hastings gave a counterexample for the case $p=1$.
It is also of probabilistic nature, but uses a very different and less canonical measure.

In our previous paper \cite{collins-nechita},
we introduced a graphical model 
that allowed us to understand more systematically the
computation of expectation and covariance of 
random channels and their powers.
In particular we studied at length the 
output of the Bell state under a random conjugate bi-channel
and obtained the explicit asymptotic behavior of this
random matrix.

In the present paper, 
we focus on the mono-channel case.
Our main result  is Theorem \ref{thm:main-thm}. It relies on a result obtained by
one author in \cite{collins-ptrf} and can be stated 
as follows:

\begin{theorem}
Let $k$ be an integer and $t$ be a real number in $(0,1)$. 
Let $\Phi_n$ be a sequence of random channels 
defined according to Equation \eqref{defphi}.
Then there exists a probability vector $\beta^{(t)}$ 
(defined in Equation \eqref{def-beta-t})
such that, for all $\e >0$,
almost surely when $n \to \iy$, for all input density matrix $\rho$, the inequality 
\begin{equation}
\spec(\Phi(\rho)) \prec \beta^{(t)}
\end{equation}
is $\e$-close to being satisfied.
Moreover, $\beta^{(t)}$ is optimal in the sense that any other probability vector 
$\beta \in \Delta_k$ satisfying the same property
must satisfy $\beta^{(t)} \prec \beta$.
\end{theorem}

We combine this result result with  bi-channel bounds 
to obtain new counterexamples to the
additivity conjectures for $p>1$. 

An illustration of our result is Corollary \ref{cor-explicite}, which one can reformulate  as follows:

\begin{theorem}
For each $p>1$ and each finite quantum space $A$ of dimension $k'\geq 2$, there exists 
an integer such that for each quantum system $B$ of 
dimension larger than this integer, 
the quantum channel arising from a quantum coupling $A$ and $B$ (of appropriate relative dimension,
depending on $p$ and $k'$)
has a high probability to be R\'enyi superadditive
when coupled with its conjugate.
\end{theorem}

From a quantum information theory point of view, the true novelty of this result is that 
any dimension $k'\geq 2$ is acceptable. This result does not seem
to be attainable 
with the alternative proofs available in \cite{hayden,hastings,brandao-horodecki,fukuda-king}.

Our techniques rely on \emph{free probability theory}. They allow us to understand entanglement
of random subspaces, 
and do not rely on a specific choice
of a measure of entanglement. 
Even though the von Neumann entropy is the most natural measure of entanglement in general, this
subtlety is important as
the papers \cite{aubrun-nechita, aubrun-nechita2} imply that
all  the $p\geq 1$ R\'enyi entropies don't enclose enough
data to fully understand entanglement. 

Our paper is organized as follows. We first recall a few basics and useful results of free probability theory of random matrix theoretical flavor in Section \ref{sec:free-proba}.
 In Section \ref{sec:quantum-channels}, we describe the random quantum channels we study.
 Section \ref{sec:confining} describes the behavior of the eigenvalues
 of the outputs of random channels.
 In Section \ref{sec:cex}, we use results of the  previous sections and of \cite{collins-nechita} to obtain new counterexamples to the additivity conjectures.

\section{A reminder of free probability}
\label{sec:free-proba}

The following is a summary of results contained in
\cite{collins-ptrf},
\cite{voiculescu}, \cite{voiculescu-dykema-nica}
and \cite{collins-sniady}.

\subsection{Asymptotic freeness}

A {\it non-commutative probability space \it} is
an algebra $\mathcal A$ with unit endowed with a tracial 
state $\phi$. 
An element of $\mathcal A$ is called
a (non-commutative) random variable. In this paper we shall be mostly concerned with the non-commutative probability space of \emph{random matrices} $(\M_n(L^{\iy-}(\Omega, \P)), \E[n^{-1}\trace(\cdot)])$ (we use the standard notation $L^{\iy-}(\Omega, \P) = \cap_{p\geq 1} L^p(\Omega, \P)$). 

Let $\mathcal A_1, \ldots ,\mathcal A_k$ be subalgebras of $\mathcal A$ having the same unit as $\mathcal A$.
They are said to be \emph{free} if for all $a_i\in \mathcal  A_{j_i}$ ($i=1, \ldots, k$) 
such that $\phi(a_i)=0$, one has  
$$\phi(a_1\cdots a_k)=0$$
as soon as $j_1\neq j_2$, $j_2\neq j_3,\ldots ,j_{k-1}\neq j_k$.
Collections $S_{1},S_{2},\ldots $ of random variables are said to be 
free if the unital subalgebras they generate are free.

Let $(a_1,\ldots ,a_k)$ be a $k$-tuple of selfadjoint random variables and let
$\mathbb{C}\langle X_1 , \ldots , X_k \rangle$ be the
free $*$-algebra of non commutative polynomials on $\mathbb{C}$ generated by
the $k$ indeterminates $X_1, \ldots ,X_k$. 
The {\it joint distribution\it} of the family $\{a_i\}_{i=1}^k$ is the linear form
\begin{align*}
\mu_{(a_1,\ldots ,a_k)} : \C\langle X_1, \ldots ,X_k \rangle &\to \C \\
P &\mapsto \phi (P(a_1,\ldots ,a_k)).
\end{align*}

Given a $k$-tuple $(a_1,\ldots ,a_k)$ of free 
random variables such that the distribution of $a_i$ is $\mu_{a_i}$, the joint distribution
$\mu_{(a_1,\ldots ,a_k)}$ is uniquely determined by the
$\mu_{a_i}$'s.
A family $(a_1^{n},\ldots ,a_k^{n})_n$ of $k$-tuples of random
variables is said to \emph{converge in distribution} towards $(a_1,\ldots ,a_k)$
iff for all $P\in \C \langle X_1, \ldots ,X_k \rangle$, 
$\mu_{(a_1^n,\ldots ,a_k^n)}(P)$ converges towards
$\mu_{(a_1,\ldots ,a_k)}(P)$ as $n\to\infty$. 
Sequences of random variables  $(a_1^{n})_n,\ldots ,(a_k^{n})_n$ are called \emph{asymptotically free} as $n \to \iy$
iff the $k$-tuple $(a_1^{n},\ldots ,a_k^{n})_n$ converges in distribution towards a family of free random variables.

The following result was contained in \cite{voiculescu} (see also \cite{collins-sniady}).

\begin{theorem}\label{libre}
Let $\{U^{(n)}_k\}_{k \in \N}$ be a collection of independent 
Haar distributed random matrices of $\M_n (\C )$ and $\{W^{(n)}_k\}_{k\in \N}$ be a 
set of constant matrices of $\M_n (\C )$ 
admitting a joint limit distribution as $n \to \iy$ with respect to the
state $n^{-1}\trace$.
Then the family $\{U^{(n)}_k, W^{(n)}_k\}_{k \in \N}$ admits a limit $*$-distribution $\{u_k, w_k\}_{k \in \N}$ with respect to $\E[n^{-1}\trace]$, such that $u_1$, $u_2$, \ldots, $\{w_1, w_2, \ldots\}$ are free.
\end{theorem}

\subsection{Free projectors.}

Let us fix real numbers 
$0\leq \alpha,\beta \leq 1$, and consider, for all $n$, a
selfadjoint projector $\pi_n \in \M_n (\C )$ of rank $q_{n}$ such that asymptotically
$q_{n}\sim \alpha n$ as $n\to\infty$.
Let $\pi_{n}'$ be a projector of rank $q_{n}'$ such that
$q_{n}'\sim \beta n$, and assume that it can be written under the form
$U\pi_n U^*$, where $U$ is a Haar distributed unitary random matrix
independent from $\pi_{n}$

It is a consequence of Theorem \ref{libre}, 
that $\pi_{n}$ and $\pi_{n}'$ are asymptotically free. 
Therefore $\pi_{n}\pi_{n}'\pi_{n}$ has an empirical eigenvalues
distribution converging towards a probability measure.
This measure is usually  denoted by
$\mu_1\boxtimes\mu_2$, where $\mu_1,\mu_2$ are the limit empirical eigenvalue distributions of the projectors $\pi_n$ and $\pi'_n$ respectively:
\begin{equation*}
\mu_1= (1-\alpha)\delta_0+\alpha\delta_1, \qquad \mu_2=(1-\beta)\delta_0+\beta\delta_1.
\end{equation*}

In this specific case, we can compute explicitly $\mu_1\boxtimes\mu_2$.
For this purpose, we 
 introduce two 2-variable functions which will be of great importance in what follows.
\begin{align*}
\phi^+:\{(x,y) \in [0,1]^2 \} &\to [0,1] \\
(x,y) &\mapsto 1-\left[ \sqrt{(1-x)(1-y)} - \sqrt{xy} \right]^2\\
\phi^-:\{(x,y) \in [0,1]^2 \} &\to [0,1] \\
(x,y) &\mapsto 1-\left[ \sqrt{(1-x)(1-y)} + \sqrt{xy} \right]^2
\end{align*}
Let us omit the variables of $\phi^{+/-}$ and rewrite
$$\phi^{+/-}=
\phi^{+/-}(\alpha,\beta)=\alpha +\beta-2\alpha\beta\pm\sqrt{4\alpha\beta (1-\alpha)(1-\beta)}$$
It follows then from \cite{voiculescu-dykema-nica}, Example
3.6.7, that
\begin{equation*}
\mu_1\boxtimes\mu_2=[1-\min (\alpha,\beta)]\delta_0+
[\max (\alpha +\beta-1,0)]\delta_1+
\frac{\sqrt{(\phi^+ -x)(x-\phi^-)}}{2\pi x(1-x)}1_{[\phi^-,\phi^+ ]}dx
\end{equation*}
The proof relies on a technique introduced by Voiculescu to compute
$\mu_1\boxtimes\mu_2$ in general, called the $S$-transform.
For more details, we refer the interested reader to \cite{voiculescu-dykema-nica}. Since we are only interested in $\phi^+$, we consider the two-variable function $\phi:[0,1]^2 \to [0,1]$
$$\phi(x,y) = 
\begin{cases}
0 \quad & \text{if } x=0 \text{ or } y=0;\\
\phi^+(x,y) \quad & \text{if } x,y>0  \text{ and } x+y \leq 1;\\
1 \quad & \text{if } x+y>1.\\
\end{cases}
$$

In the case where $\alpha +\beta <1$, the ranges of $\pi_n$ and $\pi'_n$ do not (generically) overlap and $\phi(\alpha,\beta)<1$.
The previous asymptotic freeness results imply that almost surely, 
$$\liminf_n ||\pi_n\pi_n'\pi_n||_{\infty}\geq \phi (\alpha,\beta ).$$
We are interested in whether we actually have
$$\lim_n ||\pi_n\pi_n'\pi_n||_{\infty}= \phi (\alpha,\beta ) <1.$$
This turns out to be true.
This is an involved result whose proof we won't discuss here. 
We just recall the result below as a theorem, following
 \cite{collins-ptrf} (Theorem 4.15), see also \cite{ledoux}.

\begin{theorem}
\label{thm:ptrf}
In $\C^n$, choose at random according to the Haar measure
 two independent subspaces $V_{n}$ and $V_{n}'$ of respective
dimensions $q_n\sim \alpha n$ and $q_n'\sim\beta n$
where $\alpha,\beta\in (0,1)$.
 Let $\pi_n$ (resp. $\pi_n'$) be the orthogonal projection onto $V_{n}$ 
 (resp. $V_{n}'$).
 Then,
  \[\lim_n \norm{\pi_{n}\pi_{n}'\pi_n}_\iy=\phi (\alpha,\beta).\]
\end{theorem}

\section{Quantum channels and additivity conjectures}
\label{sec:quantum-channels}

\subsection{R\'enyi entropies and minimum output entropies}
Let $\Delta_k = \{x \in \R_+^k \, | \, \sum_{i=1}^k x_i = 1\}$ be the $(k-1)$-dimensional probability simplex. For a positive real number $p>0$, define the \emph{R\'enyi entropy of order $p$} of a probability vector $x \in \Delta_k$ to be
\[H^p(x) = \frac{1}{1-p}\log\sum_{i=1}^k x_i^p.\]
Since $\lim_{p \to 1} H^p(x)$ exists, we define the \emph{Shannon entropy} of $x$ to be
this limit, namely:
\[H(x) = H^1(x) = -\sum_{i=1}^k x_i \log x_i.\]
We extend these definitions to density matrices by functional calculus:
\begin{align*}
H^p(\rho) &= \frac{1}{1-p}\log \trace \rho^p;\\
H(\rho) &= H^1(\rho) = -\trace \rho \log \rho.
\end{align*}

Given a vector $x\in \C^n, \norm{x} = 1$, we call 
$P_x$ the rank one orthogonal projection onto the span of $x$.
Using Dirac's bra-ket notation, $P_x=\ketbra{x}{x}$.
More generally, for a subspace $V\subset \C^n$,
we denote by $P_V$ the orthogonal projection onto 
$V$ in $\M_n(\C)$.

A \emph{quantum channel} is a linear completely positive trace preserving map $\Phi:\M_n(\C)\to \M_k(\C)$. 
The trace preservation condition means that density matrices are mapped to density matrices, and 
the complete positivity reads:
\[ \forall d \geq 1, \quad \Phi \otimes \I_d : \M_{nd}(\C) \to \M_{kd}(\C) \text{ is a positive map.}\]
We recall that according to Stinespring theorem,
a linear map $\Phi : \M_n(\C) \to \M_k(\C)$ is a quantum channel if and only if there exists a finite dimensional Hilbert 
space $\K = \C^{d}$, and a partial isometry $V \in \End(\C^n, \C^{kd})$ (satisfying $V^*V= \I_n$)
such that
\begin{equation}\label{eq:Stinespring_form}
\Phi(X) = \trace_\K\left[ V X V^* \right], \quad \forall X \in \M_n(\C).
\end{equation}

For a quantum channel $\Phi:\M_n(\C) \to \M_k(\C)$, we define its \emph{minimum output R\'enyi entropy (of order $p$)} by
\[H^p_{\min}(\Phi) = \min_{\substack{\rho\in \M_n(\C) \\ \rho\geq 0, \trace{\rho}=1}} H^p(\Phi(\rho)).\]

Since the R\'enyi entropies are concave functions, their minima are attained on the extremal points of the set of density matrices and hence 
\[H^p_{\min}(\Phi) = \min_{\substack{x\in \C^n \\ \norm{x} = 1}} H^p(\Phi(P_x)).\]

\subsection{The random quantum channel model}
\label{subsec:random-quantum-channel-model}
We  fix an integer $k$ and a real number
$t\in (0,1)$.
For each $n$, let
$U_n\in \M_{nk}(\C )$ be a random unitary matrix distributed according to the Haar measure, and
 $q_n$ be a projection of $\M_{nk} (\C )$ of trace $p_n$ such that $p_n/(kn)\sim t$ as $n\to\infty$.
 To $q_n$ we associate a non-unital matrix algebra map $\chi_n: \M_{p_n}(\C ) \to \M_{nk} (\C )$
 satisfying $\chi_n (1)=q_n$.
 The choice of $\chi_n$ is unique up to unitary conjugation, and the actual choice 
 of $\chi_n$ is irrelevant for the computations
 we want to perform - in the sense that any choice will yield the same results.

We study the sequence of random channels 
$$\Phi_n: \M_{p_n}(\C ) \to \M_k(\C)$$
given by
\begin{equation}
\label{defphi}
\Phi_n (X)=\trace_n (U_n(\chi_n (X))U_n^*).
\end{equation}
\begin{remark}
In our previous paper \cite{collins-nechita}, we considered exactly the same model of random quantum channels, with one small difference: the partial trace was taken with respect to $\C^k$. However, it is well-known that, when partial tracing a rank one projector, the non-zero eigenvalue of the resulting matrix do not depend on which space is traced out. Hence, from the point of view of eigenvalue statistics, the model we consider here is identical with the one in \cite{collins-nechita}, Section 6.2.
\end{remark}

Graphically, our model amounts to
Figure \ref{fig:quantum_channel}.
We refer to the first paper of this series, \cite{collins-nechita} for details about this
graphical notation.

\begin{figure}[ht]
\includegraphics{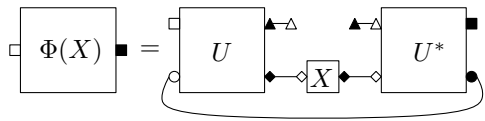}
\caption{Diagram for $\Phi(X)$}
\label{fig:quantum_channel}
\end{figure}

We are interested in the random process given by the 
\emph{set of all possible} eigenvalues of $\Phi(X)$
as $n\to\infty$. In our setup, 
we deal with $k$ eigenvalues. 

Let $V_n$ be the image of $U\chi_nU^*$.
This is a random vector space in $\C^n\otimes \C^k$ of dimension 
$p_n$ distributed according to the uniform measure on the Grassmannian space $\Gr_{p_n}(\C^{nk})$.

If we can ensure that the entanglement of \emph{every} norm one vector 
$x\in V_n$ in $\C^n\otimes \C^k$ is large with high probability
for the uniform measure on $V_n\in \Gr_{p_n}(\C^{nk})$,
 this will yield new entropy bounds. 
The entanglement is always a concave function of
the principal values of $x$.
We recall that
for an element $x \in \C^n \otimes \C^k$ we denote $\lambda(x)$ and $\rk(x)$ 
the singular values and the rank of $x$, when viewed as a matrix $x \in \M_{n \times k}(\C)$.
In quantum information theory, these quantities are also called
  the \emph{Schmidt coefficients} and the \emph{Schmidt rank} of $x$ respectively:
  \[x = \sum_{i=1}^{\rk(x)} \sqrt{\lambda_i(x)} e_i \otimes f_i,\]
  where $\{e_i\}$ and $\{f_i\}$ are orthonormal families from $\C^n$ and $\C^k$ respectively. 
  Both quantities can also be expressed as the rank and respectively the spectrum of the reduced 
  density matrix $\trace_n P_x$.
The strategy adopted in this paper is to describe a convex polyhedron 
such that with high probability, for a vector subspace $V$ chosen at random, for all input $x \in V$, the eigenvalue vector $\lambda(x)$ belongs to a neighborhood of this convex set.

\subsection{Known bounds}

Some results are already available in order to quantify the entanglement of 
generic spaces in $ \Gr_{p_n}(\C^n\otimes \C^k)$.
The best result known so far is arguably the
 following theorem of Hayden, Leung and Winter in
\cite{hayden-leung-winter}:

\begin{theorem}[Hayden, Leung, Winter, \cite{hayden-leung-winter}, Theorem IV.1]
\label{thm-hayden-leung-winter}
Let $A$ and $B$ be quantun systems of dimesion $d_a$ and $d_B$
with $d_b\geq d_A\geq 3$ Let $0<\alpha <\log d_A$. Then there exists a subspace
$S\subset A\otimes B$ of dimension
$$d\sim d_Ad_B\frac{\Gamma\alpha^{2.5}}{(\log d_A)^{2.5}}$$

such that all states $x\in S$ have entanglement satisfying
$$H(P_x)\geq \log d_A -\alpha -\beta,$$
where $\beta =d_A/(d_B\log 2)$ and $\Gamma = 1/1753$.
\end{theorem}

To prove this result, the authors require sophisticated methods from asymptotic geometry theory. In particular, they need estimates on the
covering numbers of unitary groups by balls of radius $\e$ and results of concentration of measure. 
The results of concentration of measure are applied to a specific measure of entanglement (e.g. one entropy $H^p$), 
therefore the measure of entanglement does not deal directly with the behavior of the Schmidt coefficients, but
rather with the behavior of a function of them.

\section{Confining the eigenvalues almost surely}
\label{sec:confining}

\subsection{Main result}

Our strategy is to describe a convex polyhedron $K$ inside the probability simplex $\Delta_k$ with the property that, for all $\e >0$, almost surely when $n$ goes to infinity, all input density matrices $\rho \geq 0, \trace(\rho)=1$, are mapped to output states $\Phi(\rho)$ whose spectra are contained $K+\e$, the $\e$-neighborhood of $K$.

For $t \in (0,1)$, let us first define the vector $\beta^{(t)} \in \R^k$, where  
\begin{equation}
\label{def-beta-t}
\beta^{(t)}_j = \phi\left(\frac j k,t \right) - \phi\left(\frac{j-1}{k},t\right), \quad \forall \; 1 \leq j \leq k.
\end{equation}
One can check directly that $\beta^{(t)}$ is a probability vector and that it is a non-increasing sequence. Moreover, $\beta^{(t)}_1 = \phi(1/k, t)$ and $\beta^{(t)}_j = 0$ for $j \geq \floor{k(1-t)}+2$.

Since the \emph{majorization} partial order plays an important role in this situation, let us remind here the definition and some basic properties of this relation. For two probability vectors $x, y \in \Delta_k$, 
we say that $x$ is \emph{majorized} by $y$ (and we write $x \prec y$) 
iff for all $j \in \{1, \ldots, k\}$
\begin{equation} \label{defmaj}
s_j(x) = \sum_{i=1}^{j}{x^\downarrow_i} \leq \sum_{i=1}^{j}{y^\downarrow_i} = s_j(y),
\end{equation}
where $x^\downarrow$ and $y^\downarrow$ are the decreasing rearrangements of $x$ and $y$; note that for $j=k$ we actually have an equality, since $x$ and $y$ are probability vectors. We extend the functions $s_j$, by functional calculus, to selfadjoint matrices $X \in \M_k(\C)$. The majorization relation can also be characterized in the following way: for a probability vector $y$ and a permutation $\sigma \in \S_k$, denote by $\sigma .y$ the vector obtained by permuting the coordinates of $y$ along $\sigma$ : $(\sigma . y)_i = y_{\sigma(i)}$. Then
\[x \prec y \text{ iff. } x \in S(y),\]
where $S(y)$ is the convex hull of the set $\{\sigma . y \, | \, \sigma \in \S_k\}$. Moreover, the extremal points of $S(y)$ are exactly $y$ and its permutations $\sigma.y$. In Figure \ref{fig:beta_simplex}, we plot $\Delta_3$, the 2-dimensional simplex together with the sets $S(\beta^{(t)})$, for $t=1/k'$ and $k' = 2, 3, 4, 5, 10, 20, 50, 100$. Notice that for $k' = 2,3$, the set $S(\beta^{(1/k')})$ touches the triangle $\Delta_3$, because of the fact that $\beta^{(t)}$ has in this case a zero coordinate.

\begin{figure}[ht]
\includegraphics[width=0.4\textwidth]{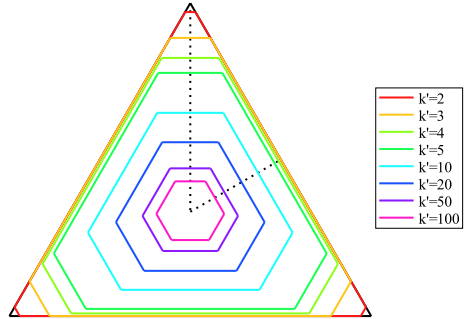}
\caption{The 2-dimensional probability simplex with the sets $S(\beta^{(t)})$, for $t=1/k'$ and $k' = 2, 3, 4, 5, 10, 20, 50, 100$.}
\label{fig:beta_simplex}
\end{figure}

We can now state the main result of this section:

\begin{theorem}
\label{thm:main-thm}
Let $t$ be a parameter in $(0,1)$ and $\e >0$.
Let $S(\beta^{(t)})+\e$ be the $\e$-ball around $S(\beta^{(t)})$ in $\Delta_k$.
Then, almost surely when $n \to \iy$, for all input density matrix $\rho$,
\begin{equation}\label{eq:main-thm}
\spec(\Phi(\rho)) \in S(\beta^{(t)})+\e.
\end{equation}
Moreover, $\beta^{(t)}$ is optimal: a probability vector $\beta \in \Delta_k$ such that, with positive probability,
\begin{equation}\label{eq:main-thm-optimal}
\spec(\Phi(\rho)) \in S(\beta)+\e \quad\forall \rho
\end{equation}
must satisfy $\beta^{(t)} \prec \beta$.
\end{theorem}

We split the proof of Theorem \ref{thm:main-thm} into several lemmas.
The first one is an easy consequence of the definition of the operator norm.

\begin{lemma}
\label{lem:elementaire}
Let $Q,R$ be two selfadjoint projections in $\M_n(\C)$. Then
\[\norm{QRQ}_{\iy}=\max_{x \in \im Q} \trace (P_xR).\]
\end{lemma}
\begin{proof} Since $QRQ$ is a self adjoint operator, we have:
\begin{align*}
\norm{QRQ}_\iy &= \sup_{\norm y \leq 1} \scalar{QRQy}{y} = \sup_{\norm y \leq 1} \scalar{RQy}{Qy} \\
&= \sup_{\substack{x \in \im Q \\ \norm x \leq 1}} \scalar{Rx}{x} = \max_{x \in \im Q} \trace(P_x R).
\end{align*}
\end{proof}

The following lemma is a reformulation of the min-max theorem:
\begin{lemma}
\label{lem:single-Fi}
Let $\lambda_1 \geq \lambda_2 \geq \cdots \geq  \lambda_k$ be the Schmidt coefficients of a vector $x \in \C^{nk}$. Then, for all $1 \leq j \leq k$,
\[s_j(x) = \lambda_1 +\lambda_2 + \cdots + \lambda_j = \max_{F \in \Gr_j(\C^k)} \trace(P_xP_{\C^n \otimes F}).\]
\end{lemma}
\begin{proof}
 
Since $\lambda_i$ are the eigenvalues of $\trace_n P_x \in \M_k(\C)$,
 the min-max theorem for $\trace_n P_x$ can be stated as:
\[s_j(x)=\max_{F\in \Gr_j(\C^k)}\trace ( P_F\trace_n P_x ).\]
The conditional expectation property of the partial trace
implies that 

\[s_j(x)=\max_{F\in \Gr_j(\C^k)} \trace(P_x \cdot \I_n  \otimes P_F)
=\max_{F\in \Gr_j(\C^k)} \trace(P_x \cdot P_{\C^n \otimes F}).\]
 \end{proof}

We are interested in majorization inequalities which hold uniformly for all norm one elements of a subspace $V$. In other words, we are interested in the quantity
\[\max_{\substack{x \in V \\ \norm{x} = 1}} s_j(x) = \max_{\substack{x \in V \\ \norm{x} = 1}} \max_{F \in \Gr_j(\C^k)} \trace(P_xP_{\C^n \otimes F}).\]
Since $k$ is a fixed parameter of our model, in order to compute the maximum over the Grassmannian, it suffices to consider only a finite number of subspaces $F$:
\begin{lemma}
\label{lem:estimate-grassmannian}
For all $\e >0$, for all $j$, there exists a finite number of $j$-dimensional subspaces $F_1, \ldots, F_N \in \Gr_j(\C^k)$ 
such that, for all $x \in \C^{nk}$,
\[\max_{i=1}^N \trace(P_x P_{\C^n \otimes F_i}) \leq s_j(x) \leq \max_{i=1}^N \trace(P_x P_{\C^n \otimes F_i}) + \e.\]
\end{lemma}

Note that in Lemma \ref{lem:estimate-grassmannian},
 $N$ does depend on $\e$ but can be chosen to be finite 
for any $\e >0$. 

\begin{proof}
We only need to prove the second inequality. Since the Grassmannian $\Gr_j(\C^k)$ is compact and metric for $d(E, F) = \norm{P_E - P_F}_{\infty}$, for all $\e >0$ there exists a covering of $\Gr_j(\C^k)$ by a finite number of balls of radius $\e$ centered in $F_1, \ldots, F_N$. Fix some $x \in \C^{nk}$ and consider the element $F \in \Gr_j(\C^k)$ for which the maximum in the definition of $s_j(x)$ is attained. $F$ is inside some ball centered at $F_i$ and we have
\[\trace(P_x P_{\C^n \otimes F}) \leq  \trace(P_x P_{\C^n \otimes F_i}) + \module{\trace(P_x (P_{\C^n \otimes F}) - P_{\C^n \otimes F_i}))} =\]
\[= \trace(P_x P_{\C^n \otimes F_i}) + \norm{P_F - P_{F_i}}_{\iy}\leq \trace(P_x P_{\C^n \otimes F_i}) + \e,\]
and the conclusion follows.
\end{proof}

Now we are ready to prove Theorem \ref{thm:main-thm}.

\begin{proof}[Proof of Theorem \ref{thm:main-thm}]
First, notice that it suffices to show \eqref{eq:main-thm} holds for rank one projectors $\rho = P_x$. The general case follows from the convexity of the functions $s_1, \ldots, s_k$. 

Let $\varepsilon >0$ and $j\in \{1,\ldots ,k\}$. For a random subspace $V \subset \C^{nk}$ of dimension $p_n\sim tnk$,
\[\max_{\substack{x \in V \\ \norm{x} = 1}} s_j(x) = \max_{\substack{x \in V \\ \norm{x} = 1}} \max_{F \in \Gr_j(\C^k)} \trace(P_xP_{\C^n \otimes F}).\]
Using the compactness argument in Lemma \ref{lem:estimate-grassmannian}, one can consider (at a cost of $\e$) only a finite number of subspaces $F$:
\[\max_{\substack{x \in V \\ \norm{x} = 1}} s_j(x) \leq \max_{i=1}^N \max_{\substack{x \in V \\ \norm{x} = 1}} \trace(P_xP_{\C^n \otimes F_i}) + \e.\]
According to Theorem \ref{thm:ptrf},
for all $i\in \{1,\ldots N\}$, almost surely when $n \to \iy$,
\[\lim_n \norm{P_V P_{\C^n\otimes F_i}P_V}_\iy=\phi (j/k, t).\]
Since $N$ is finite, with probability one, the above equality is true \emph{for all} $i$. Next, using Lemma \ref{lem:elementaire}, one has that, almost surely,
\[\limsup_n \max_{\substack{x \in V \\ \norm{x} = 1}} s_j(x)\leq \phi (j/k, t)+\varepsilon,\] 
which concludes the proof of the direct implication.

Conversely, let $\beta \in \Delta_k$ be a probability vector which satisfies Equation \eqref{eq:main-thm-optimal}. For $j \in \{1, 2, \ldots, k\}$ fixed, let $F_0$ be a subspace of $\C^k$ of dimension $j$. We have
\begin{align*}
\max_{\substack{x \in V \\ \norm{x} = 1}} s_j(x) &= \max_{\substack{x \in V \\ \norm{x} = 1}} \max_{F \in \Gr_j(\C^k)} \trace(P_xP_{\C^n \otimes F}) \\
&\geq \max_{\substack{x \in V \\ \norm{x} = 1}}  \trace(P_xP_{\C^n \otimes F_0}) = \norm{P_V P_{\C^n\otimes F_0}P_V}_\iy \xrightarrow[n \to \iy]{\text{a.s.}} \phi (j/k, t).
\end{align*}
Since, with positive probability, $\max_{\substack{x \in V \\ \norm{x} = 1}} s_j(x) \leq s_j(\beta)+\e$, we conclude that $s_j(\beta) \geq \phi (j/k, t) = s_j(\beta^{(t)})$ and the proof is complete.
\end{proof}

The interest of Theorem 
\ref{thm:main-thm} in comparison to 
Theorem \ref{thm-hayden-leung-winter} is that
it  does not rely specifically on one measurement of 
entanglement, as we are able to 
confine almost surely the eigenvalues in a convex set.
Also, our argument relies neither on 
concentration inequalities nor on
net estimates, as we fix $k$.
However, unlike Theorem
\ref{thm-hayden-leung-winter}, 
Theorem \ref{thm:main-thm}
does not give explicit control on $n$.
It is theoretically possible to give an explicit control on $n$ (using techniques
introduced in \cite{ledoux}), but 
this would lead to considerably involved technicalities.

\subsection{Application to Entropies}

Once the eigenvalues of the output of a channel have been confined inside a fixed convex polyhedron, entropy inequalities follow easily. Indeed, the confining polyhedron is defined in terms of the majorization partial order, and thus the notion of Schur-convexity (see \cite{bhatia}) is crucial in what follows.

A function $f : \R^k \to \R$ is said to be Schur-convex if $x \prec y$ implies $f(x) \leq f(y)$. 
The R\'enyi entropies $H^p$ are Schur-concave, and thus majorization relations $x \prec y$ imply $H^p(x) \geq H^p(y)$ for all $p \geq 1$. The reciprocal implication has been studied in \cite{aubrun-nechita, aubrun-nechita2}: entropy inequalities $H^p(x) \geq H^p(y)$ (for all $p \geq 1$) characterize a weaker form of majorization called \emph{catalytic majorization}, which has applications in LOCC protocols for the transformation of bipartite states.

For the purposes of this paper, the main corollary of 
Theorem \ref{thm:main-thm}  is the following
 
\begin{theorem}
\label{thm:application-entropies}
For a fixed parameter $t$, almost surely, when $n \to \iy$, for all input $\rho \in \D_{tnk}$, 
$$\liminf_n H^p_{\min}(\Phi_U) \geq H^p(\beta^{(t)}).$$
\end{theorem}

\begin{proof}
This follows directly from Theorem \ref{thm:main-thm} and from the Schur-concavity of the R\'enyi entropies.
\end{proof}

\section{New examples and counterexamples of superadditive channels}
\label{sec:cex}

Since our main result, Theorem \ref{thm:main-thm}, is valid almost surely in the limit $n \to \iy$, the limiting objects depend only on the (\emph{a priori} fixed) parameters $k$ and $t$. In what follows, we consider large values of the parameter $k$, and introduce the ``little-o'' notation $o(\cdot)$ with respect to the limit $k \to \iy$. 

\subsection{Superadditivity}

We start with a crucial recent series of result, which we summarize into the following theorem:

\begin{theorem} 
For all $p\geq 1$, there exist
quantum channels $\Phi_1$ and $\Phi_2$ such that
\begin{equation}\label{eq:additivity_H1}
H^p_{\min}(\Phi_1 \otimes \Phi_2) < H^p_{\min}(\Phi_1) + H_{\min}(\Phi_2).
\end{equation}
\end{theorem}

This theorem results mainly from the papers \cite{hastings, hayden,hayden-winter}.
Note that the equality
$$H^p_{\min}(\Phi_1 \otimes \Phi_2) = H^p_{\min}(\Phi_1) + H_{\min}(\Phi_2)$$
for any $\Phi_1$ and $\Phi_2$ and any $p>1$ was a conjecture until 2007, and that even nowadays, no concrete
counterexamples are known for $p$ small or $p=1$.

\subsection{The Bell phenomenon}

In order to provide counterexamples for the additivity conjectures, one has to produce lower bounds for the minimum output entropy of single copies of the channels (and this is where Theorem \ref{thm:main-thm} is useful) and upper bounds for the minimum output entropy of the tensor product of the quantum channels. The latter task is somewhat easier, since one has to exhibit a particular input state such that the output has low entropy. 

The choice of the input state for the product channel is guided by the following observation. It is clear that if one chooses a product input state $\rho = \rho_1 \otimes \rho_2$, then the output state is still in product form, and the entropies add up:
\[H^p([\Phi_1 \otimes \Phi_2](\rho_1 \otimes \rho_2)) = H^p(\Phi_1(\rho_1) \otimes \Phi_2(\rho_2)) = H^p(\Phi_1(\rho_1)) + H^p(\Phi_2(\rho_2)).\]
Hence, such choices cannot violate the additivity of R\'enyi entropies. Instead, one has to look at \emph{entangled} states, and the maximally entangled states are obvious candidates. 

All our examples rely on the study of the product of conjugate channels
$$\Phi_n\otimes\overline\Phi_n$$
where 
$$\Phi_n (X)=\trace_n(U_n\chi_n(X)U_n^*), \quad \overline\Phi_n (X)=\trace_n(\overline U_n\chi_n(X)U_n^t)$$
have been introduced in subsection \ref{subsec:random-quantum-channel-model}.
Our task is to obtain a good upper bound for
$$\limsup_n H_{min}^p (\Phi_n\otimes\overline\Phi_n).$$
Our strategy is systematically to write
$$\limsup_n H_{min}^p (\Phi_n\otimes\overline\Phi_n)
\leq H_{min}^p( \Phi_n\otimes\overline\Phi_n(E_{tnk}))$$
where $E_{tnk}$ is the maximally entangled state over the input space $(\C^{tnk})^{\otimes 2}$. More precisely, $E_{tnk}$ is the projection on the Bell vector
\[\Bell_{tnk} = \frac{1}{\sqrt{tnk}}\sum_{i=1}^{tnk} e_i \otimes e_i,\]
where $\{e_i\}_{i=1}^{tnk}$ is a fixed basis of $\C^{tnk}$. Using the graphical formalism of \cite{collins-nechita}, we are dealing with the diagram in Figure \ref{fig:bi_channel_UUbarmodifie} (recall that square symbols correspond to $\C^k$, round symbols correspond to $\C^n$, diamond ones to $\C^{tnk}$ and triangle-shaped symbols correspond to $\C^{1/t}$).

\begin{figure}[ht]
\includegraphics{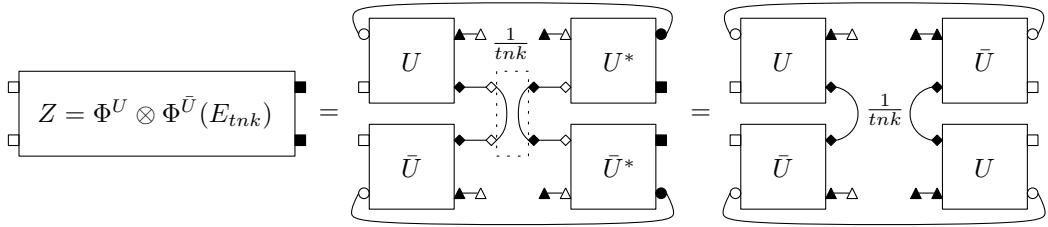}
\caption{$Z_n = \Phi^U \otimes \Phi^{\bar U} (E_{tnk})$}
\label{fig:bi_channel_UUbarmodifie}
\end{figure}

The random matrix $\Phi_n\otimes\overline\Phi_n(E_{tnk})$
was thoroughly studied in our previous paper \cite{collins-nechita}
and we recall here one of the main results of this paper:

\begin{theorem}\label{thm:eig_product}
Almost surely, as $n \to \iy$, the random matrix $\Phi_n\otimes\overline\Phi_n(\Bell_{tnk}) \in \M_{k^2}(\C)$ has eigenvalues
\[\gamma^{(t)} = \left( t + \frac{1-t}{k^2},\underbrace{\frac{1-t}{k^2}, \ldots, \frac{1-t}{k^2}}_{k^2-1 \text{ times}}\right).\]
\end{theorem}

From this we deduce the following corollary, which gives an upper bound for the 
minimum output entropy for the product channel $\Phi\otimes\overline\Phi$:

\begin{corollary}
\label{cor:upperbound-conjugate}
Almost surely, as $n \to \iy$, 
\[\limsup_n H^p_{\min}(\Phi_n\otimes\overline\Phi_n)\leq \frac{1}{1-p}\log \left[\left(t+\frac{1-t}{k^2}\right)^p
+(k^2-1)\left(\frac{1-t}{k^2}\right)^p\right]\]
In the case $p=1$ the upper bound simplifies to 
\[\limsup_n H_{\min}(\Phi_n\otimes\overline\Phi_n)\leq 
-\left(t+\frac{1-t}{k^2}\right)\log \left(t+\frac{1-t}{k^2}\right)-(k^2-1)\frac{1-t}{k^2}\log \left(\frac{1-t}{k^2}\right).\]
\end{corollary}

\subsection{Macroscopic counterexamples for the R\'enyi entropy}

In this section, we start by fixing $t=1/2$.
We assume that $k$ is even, in order to avoid non-integer dimensions. A value of $1/2$ for $t$ means that the environment to  which the input of the channel is coupled is 2-dimensional, i.e. a single qubit. 
The main result of this section is that we obtain a violation of the R\'enyi entropy 
in this simplest purely quantum case, $k'=2$.

Using Theorem \ref{thm:eig_product}, the asymptotic eigenvalue vector for the output of the product channel is 
\[\gamma = \gamma^{(1/2)} = \left(\frac{1}{2}+\frac{1}{2k^2}, \frac{1}{2k^2}, \ldots, \frac{1}{2k^2}\right).\]
The series expansion for $H^p(\gamma)$ when $k \to \iy$ and the Corollary \ref{cor:upperbound-conjugate} imply that, almost surely,
\begin{equation}
\label{eq:case-t-half}
\limsup_n H^p_{\min}(\Phi \otimes \overline \Phi) \leq \frac{p}{p-1} \log 2 + o(1).
\end{equation}

In the case of a single channel, since $\phi(x, 1/2) = 1/2 + \sqrt{x(1-x)}$, the vector $\beta = \beta^{(1/2)}$ is has a particularly simple form in this case:
\begin{align*}
\beta_1 &= \frac 1 2 + \frac{\sqrt{k-1}}{k},\\
\beta_j &= \psi\left(\frac j k\right) - \psi\left(\frac {j-1} k\right), \quad \forall \; 2 \leq j \leq k/2,\\ 
\beta_j &= 0, \quad \forall \;  k/2 < j \leq k, 
\end{align*}
where $\psi(x)= \sqrt{x(1-x)}$. Note that the first eigenvalue is large (of order 1/2) and that the others are small:
\[\beta_j \leq \frac{1}{k} \psi'\left( \frac 1 k \right) \leq \frac{1}{\sqrt k}, \quad \forall \; 2 \leq j \leq k/2.\]

\begin{theorem}\label{thm:violation_renyi}
Almost surely as $n\to \infty$, 
\[\liminf_n H^p_{\min}(\Phi )=\liminf_n H^p_{\min}(\overline\Phi )\geq  \frac{p}{p-1} \log 2 +o(1).\]
Since
$$\limsup_n H^p_{\min}(\Phi \otimes \overline \Phi) \leq \frac{p}{p-1} \log 2 + o(1),$$
the additivity of the R\'enyi $p$-norms is violated for all $p>1$. 
\end{theorem}

\begin{proof}
We shall provide a lower bound for $H^p(\beta)$. Notice that the main contribution is given by the largest eigenvalue: $\beta_1^p = 2^{-p} + o(1)$. Next, we show that the contribution of the smaller eigenvalues is asymptotically zero. We consider three cases: $p>2$, $p=2$ and $1 < p < 2$. If $p>2$, then
\[\sum_{j\geq 2} \beta_j^p \leq \sum_{j\geq 2} k^{-p/2} \leq k^{1-p/2} = o(1).\]
For $p=2$, one has:
\begin{align*}
\sum_{j\geq 2} \beta_j^2 &= \sum_{j= 2}^{k/2} \left[\psi(\frac{j}{k}) - \psi(\frac{j-1}{k})\right]^2 \leq \sum_{j=2}^{k/2} \left[\frac{1}{k} \cdot \sup_{(j-1)/k \leq x \leq j/k} \psi'(x)\right]^2 \\
&= \sum_{j=2}^{k/2} \left[\frac{1}{k}  \psi'\left(\frac{j-1}{k}\right)\right]^2 = \frac{1}{k} \sum_{j=1}^{k/2-1}  \frac{(1-2j/k)^2}{4j(1-j/k)} = o(1).
\end{align*} 
The case $1<p<2$ is more involved:
\begin{align*}
\sum_{j\geq 2} \beta_j^p & \leq \sum_{j=2}^{k/2} \left[\frac{1}{k}  \psi'\left(\frac{j-1}{k}\right)\right]^p 
&\leq k^{1-p} \left[ \int_0^{1/2} \psi'(t)^p dt \right] = o(1).
\end{align*}
Hence, in all three cases, $H^p(\beta) \geq \frac{p}{p-1} \log 2 + o(1)$. This inequality and Eq.~ (\ref{eq:case-t-half}) provide the announced violation of the additivity conjecture for R\'enyi entropies.
\[H^p_{\min}(\Phi_n \otimes \overline \Phi_n) \leq \frac{p}{p-1} \log 2  < 2 \cdot \left[\frac{p}{p-1} \log 2 + o(1) \right]\leq 2 H^p_{\min}(\Phi_n).\]

\end{proof}

Let us now come back to the more general case of arbitrary $t \in (0,1)$ fixed. 
It is natural to ask whether the bound $H^p(\beta^{(t)})$ is optimal.
Even though this is an open question for fixed $k$, the corollary below implies that it is asymptotically optimal for large $k$.
More precisely, let
$\Phi_{k,n}$ be the random quantum channel $\Phi_n$ introduced in Section \ref{subsec:random-quantum-channel-model} (since $k$ will vary in the statement below, we need to keep track of it). We can then state the following

\begin{corollary}
For all $p>1$, there exists a sequence $n_k$ tending to infinity as $k$ tends to infinity, such that, almost surely
\[ \lim_k H^p_{\min}(\Phi_{k,n_k} \otimes \overline \Phi_{k,n_k}) = \lim_k H^p_{\min}(\Phi_{k,n_k} )= \frac{p}{1-p} \log t.\]
In particular this means that we can almost surely estimate the Schatten $S_1\to S_p$ norm of that quantum channel:
\[ \lim_k ||\Phi_{k,n_k} \otimes \overline \Phi_{k,n_k}||_{S_1\to S_p} = \lim_k ||\Phi_{k,n_k}||_{S_1\to S_p}=t\]
\end{corollary}

\begin{proof}
For $t=1/2$, this follows directly by a diagonal argument from Theorem
\ref{thm:violation_renyi} and Equation
\eqref{eq:case-t-half}
together with the simple fact that the entropy increases when one takes tensor products:
$$H^p_{\min}(\Phi_{k,n_k} )\leq H^p_{\min}(\Phi_{k,n_k} \otimes \overline \Phi_{k,n_k}).$$
The asymptotic estimates of Theorem \ref{thm:violation_renyi} are readily adapted to arbitrary $t \in (0,1)$.
As for the norm estimate, it follows from the definition of the Schatten norm and the R\'enyi entropy, as well as
the fact that the $S_1\to S_p$ norm is attained on density matrices. 
\end{proof}

It is remarkable that the norm estimate for 
$||\Phi \otimes \overline \Phi ||_{S_1\to S_p}$
given by $||\Phi \otimes \overline \Phi(E_{tnk})||_{S_p}$
is actually optimal.
The above corollary stands as a mathematical evidence that the Bell states asymptotically maximize the
$S_1 \to S_p$ norm of $\Phi \otimes \overline \Phi$.

The first example of `R\'enyi superadditive' quantum channel was obtained by Holevo and Werner in \cite{holevo-werner} using a deterministic channel. However, their example violated the additivity conjecture only for $p>4.79$. 
Hayden and Winter found a class of random counter examples for the whole range of parameters $p>1$ in \cite{hayden}. 
Our being able to prescribe $t$ in the counterexample of Theorem \ref{thm:violation_renyi}
is an improvement to the counterexamples provided in 
the paper \cite{hayden-winter} (even though there is evidence  that the very recent techniques
of \cite{fukuda-king,brandao-horodecki} could be applied for $p>1$ and \emph{finite} $t$ -- yet perhaps
not as big as $1/2$ or $1/3$). 

Physically, this means that to obtain a counterexample, it is enough to couple randomly the input to a qubit ($k'=1/t=2$) to obtain a counterexample.
The above reasoning applies actually for any $t$. In the following corollary we focus on the case $t=1/k'$ for integer $k'$, as it is more relevant physically.

\begin{corollary}
\label{cor-explicite}
For each $p>1$ and each integer $k'\geq 2$, let $t=1/k'$. There exists an integer $k_0$ such that for all $k\geq k_0$,
one has almost surely 
$$\limsup_n H^p_{\min}(\Phi \otimes \overline \Phi) < 2\liminf_n H^p_{\min}(\Phi )$$
\end{corollary}

Since the proof is very similar to the case $k'=2$, instead of providing the details, we plot in Figure \ref{fig:plots} \emph{acceptable} values for $k_0$ as functions of $p$, for several values of $k'=1/t$:
$$k_0(t, p) = \min \{k \in \N \; | \; \limsup_n H^p(\Phi \otimes \overline \Phi)(E_{tnk}) < 2 H^p(\beta^{(t)})\}.$$
Note that $k_0$ as defined above 
may note be  the \emph{smallest} dimension yielding a violation of $p$-R\'enyi additivity. It may be that a better choice for the input state of the product channel could yield a smaller value for $H^p_{\min}(\Phi \otimes \overline \Phi)$. As the plots suggest, the values of $k_0$ are not bounded when $p\to 1$. This fact is independent on the choice of the parameter $t=1/k'$. The results of \cite{brandao-horodecki, fukuda-king} suggest that there should be a $k'$ large enough for which it is possible to keep $k_0$ bounded as $p\to 1$. This improvement is due to their better bounds on $H^p_{\min}(\Phi )$, obtained using the techniques developed by Hastings in \cite{hastings}.

\begin{figure}[ht]
\centering
\subfigure[]{\includegraphics[width=0.45\textwidth]{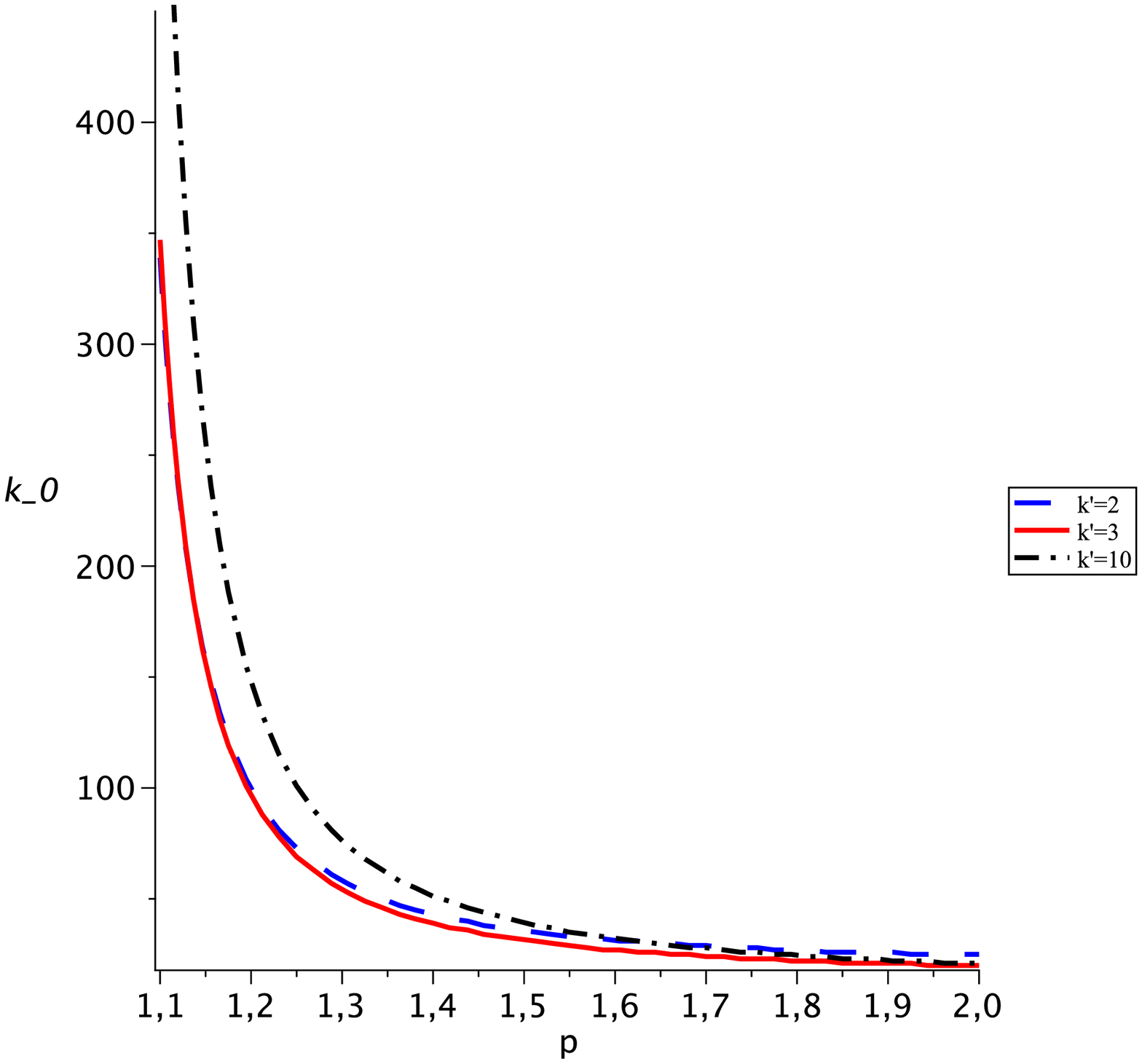}}
\subfigure[]{\includegraphics[width=0.45\textwidth]{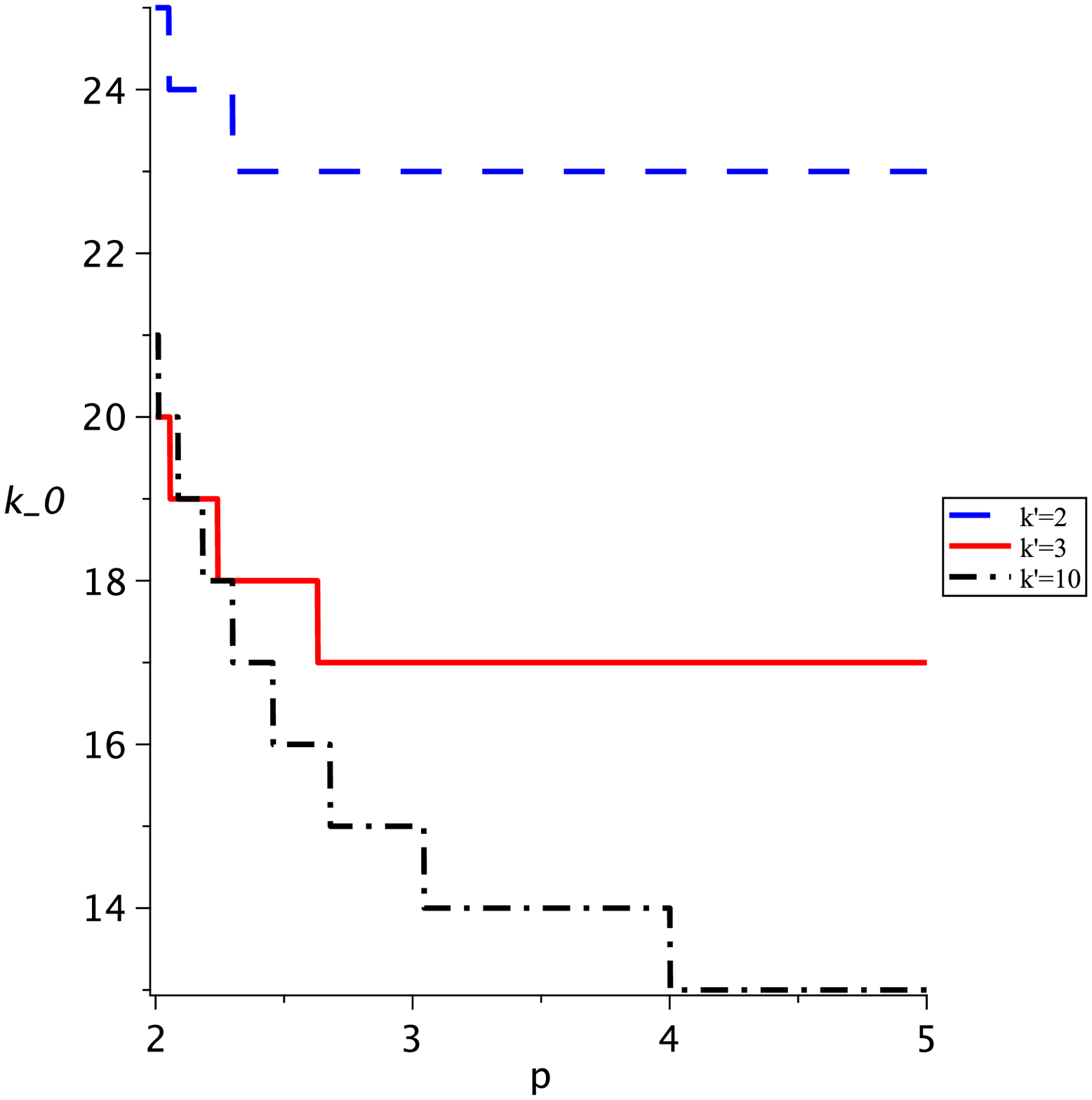}}
\caption{Plots of $k_0$, i.e. acceptable values of $k$ for which we get an asymptotic additivity violation, in function of $p$, for different values of $k' = 1/t$. Two ranges for $p$ are plotted separately: $p \in [1.1, 2]$ in (a) and  $p \in [2,5]$ in (b).}
\label{fig:plots}
\end{figure}

We finish this section by a computation showing that the above bounds are not good enough to obtain the violation of the additivity conjecture in the case $p=1$.
We start with the entropy of the product channel:
\[H(\gamma) = \log k  + \log 2 + o(1).\]
For the case of the single channel, we need an upper bound for $H(\beta^{(1/2)})$ (recall that $h(x) = -x \log x$):
\begin{align*}
H(\beta) &= h(\beta_1) + \sum_{j=2}^{k/2} h(\beta_j)
= \frac{\log 2 }{2} + \sum_{j=2}^{k/2} h\left[\psi\left(\frac{j}{k}\right) - \psi\left(\frac{j-1}{k}\right)\right] + o(1)\\
&\leq \frac{\log 2}{2} + \sum_{j=2}^{k/2} h\left[\frac{1}{k} \cdot \sup_{(j-1)/k \leq x \leq j/k} \psi'(x)\right] + o(1)
= \frac{\log 2}{2} + \sum_{j=2}^{k/2} h\left[\frac{1}{k} \psi'\left(\frac{j-1}{k}\right)\right] + o(1).
\end{align*}
Using 
\[h\left[\frac{1}{k} \psi'\left(\frac{j}{k}\right)\right] = \frac{\log k}{k} \psi'\left(\frac{j}{k}\right) + \frac{1}{k}[h \circ \psi']\left(\frac{j}{k}\right)\]
and 
\[\int_0^{1/2} [h \circ \psi'](t) \; dt = - \frac{\log 2}{2},\]
we obtain
\begin{align*}
H(\beta^{(1/2)})&\leq \frac{\log 2}{2} + \sum_{j=1}^{k/2-1} \frac{\log k}{k} \psi'\left(\frac{j}{k}\right) + \sum_{j=1}^{k/2-1} \frac{1}{k} [h \circ \psi']\left(\frac{j}{k}\right) + o(1)\\
&= \frac{\log 2}{2} + \frac{1}{2} \log k - \frac{\log 2}{2} + o(1) 
= \frac{\log k}{2} + o(1).
\end{align*}
At the end, the entropy deficit is 
\[H(\gamma) - 2H(\beta) \geq \log 2  + o(1) > 0,\]
which does not yield a violation of the minimum output von Neumann entropy.

\subsection{The case $t=k^{-\alpha}$}
We conclude this paper with the study of the case $t=k^{-\alpha}$, where $\alpha > 0$ is a fixed parameter. 
This corresponds to an exploration of a larger environment size 
$\C^{k^\alpha}$. 
To simplify the computations, we consider only the case of the minimum output von Neumann entropy.
As before, we provide estimates, when $k$ is fixed but large, for the minimum output entropies of $\Phi \otimes \ol \Phi$ and $\Phi$. 

We start with the simpler case of the product channel $\Phi \otimes \ol \Phi$. Theorem \ref{thm:eig_product} provides the almost sure eigenvalues of $[\Phi \otimes \ol \Phi](E_{tnk})$:
\[\gamma = \gamma^{(k^{-\alpha})} = \left(\frac{1}{k^\alpha} + \frac{1}{k^2} - \frac{1}{k^{\alpha+2}},\underbrace{\frac{1}{k^2} - \frac{1}{k^{\alpha+2}} , \ldots, \frac{1}{k^2} - \frac{1}{k^{\alpha+2}}}_{k^2-1 \text{ times}}\right).\]
Using the series expansion $h(1-x) = x -x^2/2 + o(x^2)$, one can compute the asymptotics for the minimum output entropy:
\begin{proposition}
For the product channel $\Phi \otimes \ol \Phi$, the following upper bounds hold almost surely:
\begin{equation}
\label{eq:conjugate-channel-t=k-alpha}
H_{\min}(\Phi \otimes \ol \Phi) \leq H(\gamma) = 
\begin{cases}
2 \log k - \frac{(2-\alpha)\log k}{k^{\alpha}} + o\left( \frac{\log k}{k^{\alpha}} \right) \quad & \text{ if } \quad 0<\alpha < 2;\\
2 \log k - \frac{2 \log 2 - 1}{k^2} + o\left( \frac{1}{k^2} \right) \quad & \text{ if } \quad \alpha = 2;\\
2 \log k - \frac{1}{2k^{2\alpha-2}} + o\left( \frac{1}{k^{2\alpha-2}} \right) \quad & \text{ if } \quad \alpha > 2.
\end{cases}
\end{equation}
\end{proposition}

Our estimate for the single channel case is as follows:
\begin{proposition}
\label{thm:estimate-k-alpha}
For all $\alpha >0$, the following lower bound holds true almost surely:
\[H_{\min}(\Phi ) \geq H_{\min}(\beta) = \log k  - \frac{\log k}{k^\alpha} + o\left( \frac{\log k}{k^\alpha}\right).\]
\end{proposition}

For the purposes of this proof, we define 
$\phi_k : [0, 1-k^{-\alpha}] \to [0, 1]$, $\phi_k(x) = \phi(x, k^{-\alpha})$ and $h(x) = -x\log x$. We have
\begin{align*}
\frac{\partial}{\partial x} \phi(x,y) &=1-2y +\sqrt{y(1-y)}\left[ \frac{\sqrt{1-x}}{\sqrt{x}} - \frac{\sqrt{x}}{\sqrt{1-x}}\right] = (1-2y)\left( 1 + \frac{g(x)}{g(y)}\right),
\end{align*}
where the function $g: (0,1) \to \R$ is defined by
\begin{align*}
g(x) = \frac{\sqrt{1-x}}{\sqrt{x}} - \frac{\sqrt{x}}{\sqrt{1-x}}.
\end{align*}

\begin{proof}
According to Theorem \ref{thm:application-entropies}, 
for all $\e >0$, 
\[H_{\min}(\Phi) \geq H(\beta) - \e,\]
where $\beta = \beta^{(k^{-\alpha})}$ is the $k$-dimensional vector defined by
\begin{align*}
\beta_1 &= \phi_k\left(\frac{1}{k}\right) ;\\
\beta_j &= \phi_k\left(\frac{j}{k}\right) - \phi_k\left(\frac{j-1}{k}\right) = \frac{1}{k} \phi'_k\left(\frac{\xi_j}{k}\right), \quad \forall \; 2 \leq j \leq J;\\
\beta_j &= 0 \quad \forall \; J < j \leq k.
\end{align*}

The index $J$ is the number of non-trivial inequalities we get by using Theorem \ref{thm:main-thm}, and it is equal to $k-1$ if $\alpha \geq 1$ and to $\floor{k - k^{1-\alpha}}$ if $\alpha < 1$.

Our purpose in what follows is to provide a ``good'' estimate for $H(\beta)$. We start by rescaling the eigenvalues: $H(\beta) = \log k + \frac{1}{k} H(k\beta)$. In this way, we can focus on the ``entropy defect'' $\log k - H(\beta)$ and reduce our problem to showing that 
\begin{equation}
\label{eq:taylor-0}
\frac{k^{\alpha-1}}{\log k} H(k\beta) = \frac{k^{\alpha-1}}{\log k} \sum_{j=1}^{J} h(k\beta_j) \xrightarrow[k \to \iy]{} -1,
\end{equation}

The next step in our asymptotic computation is to replace the unknown points $\xi_j$ by simpler estimates of the type $j/k$. Notice that the largest eigenvalue $\beta_1$ is of order $k^{-1}$. By the continuity of the function $h$, there exists a constant $C >0$ such that $|h(k\beta_j)| \leq C$ and thus, individual terms in the sum (\ref{eq:taylor-0}) have no asymptotic contribution. Moreover, we can assume $J=k-1$, ignoring at most $k^{1-\alpha}$ terms which have again no asymptotic contribution. It is clear that the function $x \mapsto \phi'_k(x)$ is decreasing at fixed $k$ and since the entropy function $h$ is increasing for $x\in (0, e^{-1})$ and decreasing for $x \geq e^{-1}$,we can bound $h(\phi'_k(\xi_j / k))$ by $h(\phi'_k(j / k))$, and we reduce our problem to showing that 
\begin{equation}
\label{eq:taylor}
\frac{k^{\alpha - 1}}{\log k} \sum_{j=1}^{k-1} h\left(\phi'_k\left(\frac{j}{k}\right)\right) \xrightarrow[k \to \iy]{} -1,
\end{equation}
or, equivalently, 
\[\frac{k^{\alpha - 1}}{\log k} \sum_{j=1}^{\floor{k/2}} h\left(\phi'_k\left(\frac{j}{k}\right)\right) + h\left(\phi'_k\left(1-\frac{j}{k}\right)\right) \xrightarrow[k \to \iy]{} -1.\]

Now, 
\[h\left[(1-2y)\left(1+\frac{g(x)}{g(y)}\right)\right] + h\left[(1-2y)\left(1-\frac{g(x)}{g(y)}\right)\right] = \]
\[=2h(1-2y) + (1-2y)\left[h\left(1+\frac{g(x)}{g(y)}\right) + h\left(1-\frac{g(x)}{g(y)}\right)\right].\]
The term $2h(1-2y) = 2h(1-2k^{-\alpha}) \sim 4k^{-\alpha}$ has no asymptotic contribution and, using $h(1+t) + h(1-t) = -t^2 + O(t^4)$, we are left with computing the limit of the main contribution 
\[\frac{k^{\alpha - 1}}{\log k} \sum_{j=1}^{\floor{k/2}} -\frac{g(j/k)^2}{g(y)^2} \sim \frac{k^{\alpha - 1}}{\log k} \sum_{j=1}^{\floor{k/2}} -g(j/k)^2k^{-\alpha}.\]
Finally, 
\[\frac{1}{k\log k} \sum_{j=1}^{\floor{k/2}} -g(j/k)^2 = \frac{-1}{\log k} \sum_{j=1}^{\floor{k/2}} \frac{1}{j} \frac{(1-2j/k)^2}{1-j/k} \sim -\frac{\log(k/2)}{\log k} \xrightarrow[k \to \iy]{} -1.\]
The error term 
\[\frac{k^{\alpha - 1}}{\log k} \sum_{j=1}^{\floor{k/2}} -\frac{g(j/k)^4}{g(y)^4} \sim \frac{k^{\alpha - 1}}{\log k} \sum_{j=1}^{\floor{k/2}} -g(j/k)^4k^{-2\alpha} \sim \frac{1}{k^{\alpha+1}\log k} \sum_{j=1}^{\floor{k/2}} \frac{1}{j^2} \frac{(1-2j/k)^4}{(1-j/k)^2} \]
converges to zero. In conclusion, we have shown that Equation \eqref{eq:taylor-0} holds and we deduce that 
\[H(\beta) = \log k - \frac{\log k}{k^\alpha} + o\left( \frac{\log k}{k^\alpha}\right).\].
\end{proof}

The bounds obtained in this section do not yield a violation of the Holevo additivity conjecture.
However, after the first version of this paper was released, Brandao-Horodecki \cite{brandao-horodecki} 
and Fukuda-King \cite{fukuda-king}
used the same model as ours and adapted original ideas from Hastings \cite{hastings} to prove that 
this model can also lead to a violation of the minimum output entropy additivity.

The techniques in \cite{brandao-horodecki,fukuda-king} yield more information on the possibility of large values of the minimum output entropy for the model under discussion. However, our proofs are of free probabilistic nature and yield results of almost sure nature. In addition, \cite{brandao-horodecki,fukuda-king} rely very much on the actual properties of Shannon's entropy function, whereas our techniques attack directly the question of the behavior of the eigenvalues. 

We conjecture that the set $S(\beta^{(t)})$ (having the property that for any $\e >0$, $S(\beta^{(t)})+\e$ contains almost surely the eigenvalues of outputs of random quantum channels)
can be made smaller and actually optimal, thus yielding as a byproduct that 
all the values $H_{\min}(\Phi)$ converge almost surely. However, the results of this paper show that
the notion of majorization is not sufficient to achieve this goal.

\section*{Acknowledgments}

This paper was completed while one author (B.C.)
 was visiting the university of Tokyo and then the university
of Wroc\l{}aw and he thanks
these two institutions
for providing him with a very fruitful working environment. I.N. thanks Guillaume Aubrun for useful discussions.

B.C. was partly funded by ANR GranMa and ANR Galoisint.
The research of both authors was supported in part by NSERC
grants including 
grant RGPIN/341303-2007.

\end{document}